\documentclass[12pt]{article}
\usepackage[utf8]{}
\usepackage{amssymb,amsmath,amsthm,graphicx,xcolor,mathtools,enumerate,mathrsfs}
\usepackage{fullpage,enumitem}
\usepackage{thmtools}
\usepackage{thm-restate}
\usepackage{hyperref}

\usepackage{hyperref}
\hypersetup{colorlinks, linkcolor={black}, citecolor={black}, urlcolor={blue!50!black}}

\usepackage{cleveref}
\theoremstyle{plain}
\newtheorem{theorem}{Theorem}[section]
\crefname{theorem}{Theorem}{Theorems}

\newtheorem{proposition}[theorem]{Proposition}
\crefname{proposition}{Proposition}{Propositions}

\crefname{corollary}{Corollary}{Corollaries}

\newtheorem{lemma}[theorem]{Lemma}
\crefname{lemma}{Lemma}{Lemmas}

\newtheorem{conjecture}[theorem]{Conjecture}
\crefname{conjecture}{Conjecture}{Conjectures}

\newtheorem{problem}[theorem]{Problem}
\crefname{problem}{Problem}{Problem}

\crefname{claim}{Claim}{Claims}

\crefname{observation}{Observation}{Observations}

\crefname{setup}{Setup}{Setups}

\crefname{fact}{Fact}{Facts}

\crefname{algorithm}{Algorithm}{Algorithms}

\crefname{remark}{Remark}{Remarks}

\crefname{example}{Example}{Examples}

\theoremstyle{definition}

\crefname{definition}{Definition}{Definitions}

\crefname{construction}{Construction}{Constructions}

\newtheorem{question}[theorem]{Question}
\crefname{question}{Question}{Questions}

\numberwithin{equation}{section}

\definecolor{DarkDesaturatedBlue}{HTML}{3A3556}
\definecolor{VividOrange}{HTML}{F15918}
\definecolor{PureOrange}{HTML}{FFBA00}
\definecolor{LightGrayishPink}{HTML}{EEC5D5}
\definecolor{VerySoftBlue}{HTML}{B5AFDB}

\usepackage{tikz}
\usetikzlibrary{math}
\usetikzlibrary{calc,decorations.pathmorphing,through}
\tikzset{snake it/.style={decorate, decoration=snake}}

\definecolor{DarkDesaturatedBlue}{HTML}{3A3556}
\definecolor{VividOrange}{HTML}{F15918}
\definecolor{PureOrange}{HTML}{FFBA00}
\definecolor{LightGrayishPink}{HTML}{EEC5D5}
\definecolor{VerySoftBlue}{HTML}{B5AFDB}
\DeclareMathOperator{\red}{red}
\DeclareMathOperator{\blue}{blue}
\DeclareMathOperator{\bin}{Bin}
\DeclareMathOperator{\tc}{tc}
\DeclareMathOperator{\tp}{tp}
\def\eps{\varepsilon}

\renewcommand{\subset}{\subseteq}

\title{Monochromatic partitions in 2-edge-coloured bipartite graphs
}

\author{Camila Fern\'andez\thanks{Laboratoire de Probabilit\'e, Statistique et Mod\'elisation (LPSM), Sorbonne University, 4 place Jussieu, 75005, Paris, France. Supported by FONDECYT Regular Grant 1180830 and by CMM CONICYT PIA AFB170001. {\it Email:} camila.fernandez.1@sorbonne-universite.fr} \and Mat\'ias Pavez-Sign\'e\thanks{Center for Mathematical Modeling (CNRS IRL2807), University of Chile.  Supported by ANID Basal Grant CMM FB210005. {\it Email:} mpavez@dim.uchile.cl} \and Maya Stein\thanks{ Department of Mathematical Engineering and Center for Mathematical Modeling (CNRS IRL2807), University of Chile. Supported by FONDECYT Regular Grant 1221905,  by ANID Basal Grant CMM FB210005,  and by MSCA-RISE-2020-101007705 project {\it RandNET}. {\it Email:} mstein@dim.uchile.cl}}

\begin{document}
\date{}
\maketitle

\begin{abstract}
We study two variations of the Gy\'arf\'as--Lehel conjecture on the minimum number of monochromatic components needed to cover an edge-coloured complete bipartite graph. Specifically, we show the following. 
\begin{itemize}
    \item For $p\gg (\log n/n)^{1/2}$, w.h.p.~every  $2$-colouring of the random bipartite graph $G\sim G(n,n,p)$ admits a cover of all but $O(1/p)$ vertices of $G$ using at most three vertex-disjoint monochromatic components. 
    \item For every $2$-colouring of a bipartite graph $G$ with parts of size $n$ and minimum degree $(13/16+o(1))n$, the vertices of $G$ can be covered using at most three vertex-disjoint monochromatic components.
\end{itemize}
\end{abstract}

\section{Introduction}
\subsection{History of tree covers and partitions}
An $r$-colouring of a graph $G$ is a colouring of the edges of $G$ with up to $r$ different colours. We are interested in determining the smallest number $\tp_r(G)$ ($\tc_r(G)$) such that in any $r$-colouring of $G$, the vertex set of $G$ can be partitioned (covered) using at most $\tp_r(G)$ ($\tc_r(G)$) monochromatic trees. There is a large amount of literature on this problem and its variants where trees are replaced by paths or cycles, and most of the attention has 
 focused on the case when $G$ is either the $n$-vertex complete graph $K_n$ or the complete bipartite graph $K_{n,m}$ with parts of size $n$ and $m$. 
We recommend the excellent surveys~\cite{fujita2015monochromatic,gyarfas2016vertex} for further reading.

Determining $\tc_r(K_n)$ is related to the celebrated conjecture~\cite{henderson1971permutation} of Ryser on matchings and transversals in hypergraphs (which was also conjectured by L\'ovasz~\cite{lovász1977kombinatorika}). Ryser's conjecture can be equivalently formulated as follows: The vertex set of any $r$-coloured graph $G$ can be covered by at most $(r- 1)\alpha(G)$ monochromatic trees, where $\alpha(G)$ denotes the independence number of $G$~\cite{erdHos1991vertex, gya77}. In particular, Ryser's conjecture would imply that 
$\tc_r(K_n)= r-1$ for all $n\ge 1$ and all $r\ge 2$, which is best possible for infinitely many values of $r$
(see e.g.~\cite{erdHos1991vertex}) and has been confirmed for $r\le 5$ and all $n\ge 1$ (see~\cite{duchet, gya77, tuza}). 
Further, it is easy to see that $\tc_r(K_n)\le r$, since the set of all maximal monochromatic stars having a fixed vertex $v$ as their centre clearly covers $V(K_n)$. 

Erd\H{o}s, Gy\'arf\'as, and Pyber~\cite{erdHos1991vertex} conjectured that the above bound holds even for tree {\it partitions}, that is, they conjectured that $\tp_r(K_n)=r-1$ for all $r\ge 2$. This conjecture has been confirmed for $r=2,3$~\cite{erdHos1991vertex}, and it is  known that $\tp_r(K_n)\le r$ if $n\ge n_0(r)$ is large enough~\cite{bal2017partitioning, haxell1996partitioning}. 

Let us now turn    to the complete bipartite graph $K_{n,m}$. It is easy to see that $\tc_r(K_{n,m})\le 2r-1$, as we can fix any edge $vw$ and consider the maximal monochromatic stars having $v$ or $w$ as their centre. Two of these stars are joined by the edge $vw$, and thus the total number of monochromatic trees needed to cover $K_{n,m}$ is at most $2r-1$. In~\cite{chen2012around}, this observation and the following conjecture are attributed to Gy\'arf\'as~\cite{gya77} and Lehel~\cite{lehel}.
\begin{conjecture}[Gy\'arf\'as and Lehel~\cite{gya77, lehel}]\label{conjecture:GL}
For all $n, m\ge 1$ and $r\ge 2$,  $\tc_r(K_{n,m})\le 2r-2$. 
\end{conjecture}

The conjecture has been verified for $r\le 5$ by
Chen, Fujita, Gy\'arfás, Lehel, and T\'oth~\cite{chen2012around}. 
Further, there are examples of $r$-colourings of which show that the bound of $2r-2$ in Conjecture~\ref{conjecture:GL} is best possible~\cite{chen2012around, gya77}. See Section~\ref{sec:gyalehvar} for more discussion of this conjecture.

\subsection{Tree covering problems in random graphs}
In 2017, Bal and DeBiasio~\cite{bal2017partitioning} initiated the study of covering problems in edge-colourings of the binomial random graph $G(n,p)$. They showed that, for $r\ge 2$ and $p\ll(r\log n/n)^{1/r}$, w.h.p.\footnote{We say that $G(n,p)$ satisfies a property $\mathcal P$ \textit{with high probability} (w.h.p.) if $\mathbb{P}(G(n,p)\in \mathcal P)=1-o(1)$ as $n$ tends to infinity. We say that $\hat{p}=\hat{p}(n)$ is the \textit{threshold} for the property $\mathcal P$, if $\mathbb P(G(n,p)\in \mathcal P)=1-o(1)$ when $p=\omega(\hat{p})$, and $\mathbb P(G(n,p)\in \mathcal{P})=o(1)$ if $p=o( \hat{p})$.}~there is an $r$-colouring of $G(n,p)$ which does not admit a cover with monochromatic trees whose number is bounded by some function in $r$. In 2021, Buci\'c, Kor\'andi, and Sudakov~\cite{bucic2021covering} showed that indeed $p=(\log n/n)^{1/r}$ is the threshold for the property of admitting a cover with $f(r)$  monochromatic trees for some function~$f$.

In view of the results from the previous subsection, it seems natural  to ask for the threshold of the property that $\tc_r(G(n,p))\le r$ or $\tp_r(G(n,p))\le r$. For $r=1$,  this corresponds to the threshold at which $G(n,p)$ becomes connected. Also note that, while in the deterministic setting it was conjectured that $\tp_r(K_n)\le r - 1$, it is not overly difficult to see that $r$ trees are needed if $p \le 1 - \varepsilon$, for any constant $\varepsilon >0$.
  
  Bal and DeBiasio~\cite{bal2017partitioning}
  conjectured that $\tp_r(G(n,p))\le r$ holds w.h.p.~when $p\ge (1+\varepsilon)(r\log n/n)^{1/r}$ and $\varepsilon >0$ is fixed. This was confirmed for $r=2$ by Kohayakawa, Mota, and Schacht~\cite{kohayakawa2019monochromatic}, and disproved for all $r\ge 3$ by Ebsen, Mota, and Schnitzer (see~\cite[Proposition~4.1]{kohayakawa2019monochromatic}) who showed that $\tp_r(G(n,p))\ge r+1$ holds w.h.p.~when $p\ll (\log n/n)^{1/(r+1)}$. For $r=3$, Badra\v{c} and Buci\'c~\cite{bradavc2023covering} showed that $p=(\log n/n)^{1/4}$ is indeed the threshold for $\tc_3(G(n,p))\le 3$, improving upon previous results from~\cite{bal2017partitioning, kohayakawa2021covering}.
However, as shown by Buci\'c, Kor\'andi, and Sudakov~\cite{bucic2021covering}, the threshold for $\tc_r(G(n,p))\le r$ is in general much higher, and lies somewhere between $(\log n/n)^{1/2^{r}}$ and $(\log n/n)^{\sqrt{r}/2^{r-2}}$.

A bipartite variant of these results, in the spirit of the Gy\'arf\'as--Lehel conjecture (Conjecture~\ref{conjecture:GL}) is still missing. In this setting,  the complete bipartite graph $K_{n,m}$ from Conjecture~\ref{conjecture:GL} should be replaced with a random bipartite graph $G(n,m,p)$, whose vertex set are disjoint copies of $[n]$ and $[m]$, respectively, with every possible edge between these two sets appearing independently with probabi\-li\-ty $p$.

We will focus here on the case $r=2$ and $n=m$. (See Sections~\ref{sec:unbalRandom} and~\ref{sec:moreColours} for a discussion of  other options.) Similarly as for complete graphs and $G(n,p)$, the bound given by Conjecture~\ref{conjecture:GL} does not carry over to $G(n,n,p)$, even when $p=1-o(1)$. Indeed, we  show that w.h.p.~$\tc_2(G(n,n,p))\ge 3$ when $1-p$ exceeds $3\log n/n$. 

\begin{proposition}\label{proposition:lowerbound}
For $p\le 1-{3\log n}/{n}$, w.h.p.~$\tc_2(G(n,n,p))\ge 3$.
\end{proposition}

Thus motivated, we are interested in the threshold of the property $\tc_2(G(n,n,p))\le 3$. We believe that this threshold should be the same as for the $r=2$ case of the non-bipartite setting.

\begin{conjecture}\label{conjecture:2-colours}
The threshold for $\tc_2(G(n,n,p))\le 3$ is $\hat{p}=(\log n/n)^{1/2}$.     
\end{conjecture}
Evidence for Conjecture~\ref{conjecture:2-colours} is given by the following theorem, which contains the $0$-statement and an approximate form of the $1$-statement. 
\begin{theorem}\label{thm:almost-cover}There exist positive constants $c$ and $C$ such that the following holds.
\begin{enumerate}
    \item If $p\le c\left( {\log n}/{n} \right)^{1/2}$, then w.h.p.~$\tc_2(G(n,n,p))\ge 4$.
    \item If $p\ge C\left( {\log n}/{n} \right)^{1/2}$, then w.h.p., in every 2-colouring of $G\sim G(n,n,p)$, all but at most $200/p$ vertices can be covered using at most three vertex-disjoint monochromatic trees.
\end{enumerate}
\end{theorem}

We prove Proposition~\ref{proposition:lowerbound} and Theorem~\ref{thm:almost-cover} in Section~\ref{sec:c}.

\subsection{Graphs of large minimum degree}
In~\cite{bal2017partitioning}, Bal and DeBiasio asked whether a minimum degree version of the tree covering problem exists, and proved that for every $r\ge 2$ there is a constant $\alpha_r\in (0,1)$ such that if~$G$ is an $n$-vertex graph with $\delta(G)\ge \alpha_r n$, then $\tc_r(G)\le r$. They also found an $r$-coloured $n$-vertex graph with minimum degree roughly ${rn}/(r+1)$ which cannot be covered with $r$ monochromatic trees and conjectured that this would be the worst-case scenario. 
\begin{conjecture}[Bal and DeBiasio~\cite{bal2017partitioning}]\label{conjecture:mindegree}Let $n,r\ge 2$. If $G$ is an $n$-vertex graph with 
\begin{equation}\label{equation:mindeg}\delta(G)>\frac{r}{r+1}(n-r-1),\end{equation}
then $\tc_r(G)\le r$.\end{conjecture}
Shortly afterwards, Gir\~ao, Letzter, and Sahasrabudhe~\cite{Girao} proved a strengthening of Conjecture~\ref{conjecture:mindegree} for $r=2$, showing that every $n$-vertex 2-edge-coloured graph 
of minimum degree exceeding $(2n-5)/3$
can be {\it partitioned} into at most two monochromatic components. 

We propose a minimum degree version of the Gy\'arf\'as--Lehel conjecture (Conjecture~\ref{conjecture:GL}). While in Conjecture~\ref{conjecture:GL}, the bound on the number of monochromatic components is $2r-2$, we will need to work with the bound $2r-1$, since a non-complete bipartite host graph may need this many components, as witnessed by the random case (Proposition~\ref{proposition:lowerbound}).  
\begin{question}\label{question:mindeg}
What is the smallest number $\alpha_r>0$ such that if $G$ is a spanning subgraph of $K_{n,n}$ with minimum degree at least $\alpha_r n$, then $\tc_r(G)\le 2r-1$?
\end{question}
The same question can be asked for partitioning.
\begin{question}\label{question:mindeg2}
What is the smallest number $\beta_r>0$ such that if $G$ is a spanning subgraph of $K_{n,n}$ with minimum degree at least $\beta_r n$, then $\tp_r(G)\le 2r-1$?
\end{question}

Note that $\beta_r \ge \alpha_r>1/2$ for all $r\ge 2$. For this, it suffices to consider  the graph $G$ consisting of two disjoint copies of $K_{\frac{n}{2},\frac{n}{2}}$. Give each of these a colouring that cannot be covered by less than $r$ components to see that $\tc_r(G)\ge 2r$. (Such a colouring can be obtained by taking a proper edge-colouring of $K_{r,r}$ and blowing up one of the edges).

For two colours, we show that the answer to Question~\ref{question:mindeg2} is at most $13/16+o(1)$. Moreover, we show that in this setting, the vertex set can even be {\it partitioned} into at most three monochromatic components.

\begin{theorem}
\label{theorem:bipartite}
For every $\delta>0$, there is $n_0$ such that for every $n\ge n_0$ the following holds. If~$G$ is a spanning subgraph of $K_{n,n}$ with minimum degree at least $(13/16+\delta)n$, then $\tp_2(G)\le 3$.
\end{theorem}
We prove Theorem~\ref{theorem:bipartite} in Section~\ref{sec:m}.


\section{Preliminaries}
We collect some graph theoretic notation, probabilistic inequalities and basic results on random bipartite graphs.

\subsection{Basic notation}
 Let $G$ be a graph with vertex set $V(G)$ and edge set $E(G)$.  For $A,B\subseteq V(G)$,  $G[A]$ is the graph induced by $A$, and, if $A$, $B$ are disjoint, $G[A,B]$ is the bipartite graph induced by $A$ and $B$, that is, the bipartite graph with parts $A$ and $B$, and all the edges $ab\in E(G)$ with $a\in A$ and $b\in B$. 
Write $e(G):=|E(G)|$, $e(A,B):=|E(G[A,B])|$ and  $e(A):=|E(G[A])|$. 
For  $x\in V(G)$ and $U\subseteq V(G)$, we write $N(x,U)$ for the set of neighbours of $x$ in $U$ and set $d(x,U):=|N(x,U)|$. If $U=V(G)$, we just write $N(x)$ and $d(x)$.
 When working with more than one graph, we add subscripts for the graph we are referring to, for example, $d_G(x)$ is the degree of a vertex $x$ in the graph~$G$.
 
 Given a $2$-colouring of $G$, with colours red and blue, we let $G_R$ and $G_B$ denote the subgraphs consisting of the red and blue edges, respectively. We will use the subscript $R$ and $B$ instead of $G_R$ and $G_B$ when referring to the red and blue subgraph, respectively. For example, for a vertex $x\in V(G)$, we write $d_R(x)$ and $d_B(x)$ instead of $d_{G_R}(x)$ and $d_{G_B}(x)$.

For   $a,b\in\mathbb R$ and $c>0$, we write $a=b\pm c$ to denote that $a$ satisfies $b-c\le a\le b+c$, while $a\ll b$ means that given $b$ one can choose $a$ sufficiently small so that all the following relevant statements hold. We omit   
floors and ceilings if this does not affect the proofs.

\subsection{Probabilistic inequalities}
We will use the following standard probabilistic inequalities (see~\cite{JLR2000} for instance). 
\begin{lemma}[Markov's inequality] 
\label{markov}
Let $X$ be a non-negative random variable. Then, for any $\lambda>0$, we have $\mathbb P(X\ge \lambda)\le \mathbb E X/\lambda.$
\end{lemma}
\begin{lemma}[Chernoff's bound]\label{lemma:chernoff}
Let $X$ be a binomial random variable.
\begin{enumerate}[label=\upshape(\roman{enumi})]

\item If $0<\delta<1$, then $\mathbb P(X\le (1-\delta)\mathbb EX)\le e^{-\delta^2\mathbb EX/2}.$
\item If $0<\delta\le 3/2$, then $\mathbb{P}(|X-\mathbb{E}X|\ge \delta \mathbb{E}X)\le 2e^{-\delta^2\mathbb{E}X/3}.$
\end{enumerate}
\end{lemma}
\begin{lemma}[Paley--Zygmund inequality]
\label{paley-zyg}
Let $X$ be a non-negative random variable with finite variance. Then, for any $0<\delta<1$, we have
\[\mathbb{P}(X\ge \delta\mathbb EX)\ge (1-\delta)^2\dfrac{(\mathbb
EX)^2}{\mathbb
EX^2}.\]
\end{lemma}
\subsection{Random graphs and graph properties}
In the \textit{binomial random graph} model $G(n,p)$ we consider a graph on vertex set $[n]$ where every possible edge appears independently with probability $p$. By a \textit{graph property} $\mathcal P$ we mean a collection of finite graphs, and we say that $\mathcal P$ is \textit{monotone} if $H$ is a spanning subgraph of $G$ and $H\in\mathcal P$ implies $G\in\mathcal P$.
Note that the property of $\tc_r$ being bounded from above by some fixed number $k$ is a monotone graph property.

The \textit{binomial random bipartite graph} model $G(n,m,p)$ is the bipartite graph whose vertex classes are disjoint copies of $[n]$ and $[m]$, respectively, and each possible edge appears independently with probability $p$. The terms \textit{with high probability} and \textit{threshold} are defined in the same way as for random graphs.

In the following lemma, we collect all the properties of random bipartite graphs that we will need for proving the approximate $1$-statement of Theorem~\ref{thm:almost-cover}. This lemma is a bipartite version of Lemma 2.1 in~\cite{kohayakawa2019monochromatic}, and most of the proofs are straightforward applications of Chernoff's bound. For the sake of completeness,  we include its proof in the appendix.

\begin{lemma}\label{lemma:prelim1}For every $0<\eps<1$, there exists a constant $C>0$ such that  for $p\ge  C(\log n/n)^{1/2}$, if  $G\sim G(n,n,p)$ has bipartition classes $V_1$ and $V_2$, then w.h.p.~the following properties hold.
\begin{enumerate}[label=\upshape(\roman{enumi})]
    \item For every $v,w\in V_1$ (resp. $u,w\in V_2$), we have $d(v)=(1\pm \eps)pn$ and $|N(v)\cap N(w)|=(1\pm \eps)p^2n$.
    \item For every $v\in V_2$ (resp. $v\in V_1$), $U\subseteq N(v)$ with $|U|\ge pn/100$, and $W\subseteq V_2$ (resp. $W\subseteq V_1$) with $|W|\ge 100/p$, we have $e(U,W)\ge p|U||W|/2$.
    \item For every $v\in V_2$ (resp. $v\in V_1)$ and $U\subseteq N(v)$ with $|U|\ge pn/100$, all but at most $100/p$ vertices $v'\in V_2$ (resp. $u\in V_1$) satisfy $d(v',U)\ge p^2n/200$.
    \item Every subgraph $H\subseteq G$ with $\delta(H)\ge (1/2+\eps)pn$ is connected.
\end{enumerate}
\end{lemma}
We remark that for ensuring the degree condition $d(x)=(1\pm\varepsilon)pn$ it suffices that $p\gg \log n/n$ rather than the stronger condition $p\gg (\log n/n)^{\frac{1}{2}}$.

The next lemma is crucial for the construction of a colouring that shows the $0$-statement in Theorem~\ref{thm:almost-cover}.

\begin{lemma}\label{lemma:prelim2}
Let $c>0$ be sufficiently small and let $p=c(\log n/n)^{1/2}$. Let $G\sim G(n,n,p)$ with bipartition classes $V_1$, $V_2$. Then, for $i\in[2]$, w.h.p.~there are at least $e^{-2c^2\log n}\binom{n}{2}$ pairs $\{u,v\}$ of distinct vertices from~$V_i$ such that $u,v$ have no common neighbours.   
\end{lemma}
\begin{proof}
The probability that two distinct vertices $u,v\in V_i$ have no common neighbours is $(1-p^2)^n$. So, letting $X_i$ denote the random variable counting the number of such $\{u,v\}$, we have  
\[\mathbb EX=\binom{n}{2}(1-p^2)^n\le \binom{n}{2}e^{-p^2n},\]
and thus
\begin{align*}\mathbb EX^2 & =\binom{n}{2}\binom{n-2}{2}(1-p^2)^{2n}+\binom{n}{3}p^n(1-p^2)^n+\binom{n}{2}(1-p^2)^n=(1+o(1))(\mathbb
EX)^2. 
\end{align*}
Therefore, by the Paley--Zygmund inequality (Lemma~\ref{paley-zyg}) we have 
\[\mathbb P\left(X\ge e^{-2c^2\log n}\binom{n}{2}\right)\ge \mathbb P(X\ge e^{-c^2\log n}\mathbb EX)\ge (1-e^{-c^2\log n})^2\dfrac{(\mathbb EX)^2}{\mathbb EX^2}=1-o(1).\]
\end{proof}

\section{Tree covers in random bipartite graphs}
\label{sec:c}
In this section, we prove Proposition~\ref{proposition:lowerbound} and Theorem~\ref{thm:almost-cover}. We start with Proposition~\ref{proposition:lowerbound}.

\begin{proof}[Proof of Proposition~\ref{proposition:lowerbound}]
Let $G\sim G(n,n,p)$ and let $V_1$ and $V_2$ be the bipartition classes of $G$. We first note that w.h.p.~each vertex in $V_i$ has at least two non-neighbours in $V_{3-i}$, for $i\in [2]$. Indeed, for a vertex $v\in V_1$ (say), the probability that $v$ has less than 2 non-neighbours is \[p^n+np^{n-1}\le e^{-2\log n}+ne^{-3\log n}\le 2n^{-2},\]
and thus by a union-bound the probability that there is a vertex with less than 2 non-neighbours is at most $2n\cdot 2n^{-2}=o(1)$. So, we may pick, arbitrarily, vertices $r\in V_1$ and $b\in V_2\setminus N(r)$, and set $X=V_1\setminus (N(b)\cup\{r\})$ and $Y=V_2\setminus (N(r)\cup \{b\})$. Note that both $X$ and $Y$ are non-empty as both $r$ and $b$ have at least 2 non-neighbours. 

We colour in red the edges from $r$ to $N(r)$, the edges between $X$ and $N(r)$, and the edges between $Y$ and $N(b)$. We colour in blue the edges from $b$ to $N(b)$, the edges between $N(r)$ and $N(b)$, and the edges between $X$ and $Y$. In this colouring, no two components can cover all of $V(G)$ (the monochromatic components that cover  $r$ and $b$ cannot cover $Y$), which proves that $\tc_2(G)\ge 3$.
\end{proof}

The following proposition proves Theorem~\ref{thm:almost-cover}~(i).

\begin{proposition} Let $c>0$ be sufficiently small  and let $p\le c(\log n/n)^{1/2}$. Then w.h.p. $\tc_2(G(n,n,p))\ge 4$. 
\end{proposition}
\begin{proof}
We note that by the monotonicity (of $\tc_2$ being bounded from below by $4$) we can assume that $p=c(\log n/n)^{1/2}$. Let $G\sim G(n,n,p)$ and let $V_1$ and $V_2$ be the bipartition classes of $G$. By Lemma~\ref{lemma:prelim1}, for each $i\in[2]$, w.h.p.~every vertex $v\in V_i$ satisfies $d(v)=(1\pm0.5)pn$. By Lemma~\ref{lemma:prelim2}, w.h.p.~there are at least $e^{-c^2\log n}\binom{n}{2}$ pairs of distinct vertices in $V_i$, for each $i\in[2]$, with no common neighbours. 

We now construct an edge-colouring of $G$ which can only be covered if at least four monochromatic components are used.  We choose $u_1,v_1\in V_1$ with no common neighbours. Then, we pick $u_2,v_2\in V_2\setminus (N(u_1)\cup N(v_1))$ with no common neighbours, which is possible because we have at least 
\[e^{-c^2\log n}\binom{n}{2}-(4pn)^2\ge e^{-c^2\log n}\binom{n}{2}-16c^2n\log n\]
options for $\{u_2,v_2\}$. We use red for all edges between $u_i,v_i$ and their neighbours, for $i\in [2]$, and blue for all the rest. Since $u_1, u_2, v_1, v_2$ do not belong to any blue component and lie in four distinct red components, we cannot cover $G$ with less than four components. Therefore $\tc_2(G)\ge 4$.
\end{proof}
The proof of Theorem~\ref{thm:almost-cover}~(ii) is captured in the following theorem, whose proof is inspired by the approach of Kohayakawa, Mota, and Schacht~\cite{kohayakawa2019monochromatic}.
\begin{theorem}There exists a  constant $C>0$ such that if $p\ge C\left( {\log n}/{n} \right)^{1/2}$, then w.h.p.~in every 2-colouring of $G\sim G(n,n,p)$, all but at most $200/p$ vertices of $G$ can be covered by at most three vertex-disjoint monochromatic trees.    
\end{theorem}
\begin{proof}
Let $C$ be a sufficiently large constant. Since the property of the theorem (being almost coverable by $m$ monochromatic trees) is monotone, we can assume $p=C(\log n/n)^{1/2}$.
 Let $G\sim G(n,n,p)$, with partition classes $V_1$ and $V_2$. 
 
 For $0<\eps\ll 1$, by Lemma~\ref{lemma:prelim1} we know that w.h.p.~$G$ satisfies the following properties: 
\begin{enumerate}[label=(B\arabic{enumi})]
    \item\label{lem:i} For $i\in[2]$ and  $v,w\in V_i$, we have $d(v)=(1\pm \eps)pn$ and $|N(v)\cap N(w)|=(1\pm \eps)p^2n$.
    \item\label{lem:ii} For $i\in[2]$ and subsets $U\subseteq V_i$ and $W\subseteq V_{3-i}$, with $|U|\ge 100/p$ and $|W|\ge pn/100$, we have $e(U,W)\ge p|U||W|/2$.
    \item\label{lem:iii} For $i\in[2]$ and a subset $W\subseteq V_i$ with $|W|\ge pn/100$, all but at most $100/p$ vertices $v\in V_{3-i}$ satisfy $d(v,W)\ge p^2n/200$.
    \item\label{lem:iv} Every subgraph $H\subseteq G$ with $\delta(H)\ge (1/2+\eps)pn$ is connected.
\end{enumerate}
Moreover, as we may assume that $n$ is sufficiently large, we have that
 \begin{align}\label{eq:pbound}
  \displaystyle \max \big\{ (1+ \eps)p^2n, {100}/{p}  \big\} & \leq \frac{pn}{100}.
\end{align}
Suppose we are given a red and blue edge-colouring of $G$. If there is a monochromatic spanning component, we are done, so assume otherwise. Set $V_R=\{ v \in V(G) : d_R(v) > \frac{1}{3}d(v) \}$ and $V_B=\{ v \in V(G) : d_B(v) > \frac{1}{3} d(v) \}$. We claim that
\begin{equation}\label{eq:1}
    \text{$V_R\neq \emptyset\neq V_B$.}
\end{equation}
For contradiction, assume that $V_R=\emptyset$.  Then~\ref{lem:i} implies that for every $v \in V(G)$, 
\[d_B(v) \ge \frac{2}{3}(1-\eps)pn \ge \left(\frac{1}{2} + \eps\right)pn.\]
Therefore, by~\ref{lem:iv}, the blue graph $G_B$ is connected and thus $G$ has a monochromatic spanning tree. The same argument applies to $V_B$, which completes the proof of~\eqref{eq:1}.

By~\eqref{eq:1}, 
and after  possibly swapping the names of $V_1$ and $V_2$, we can  assume that there are vertices $$r \in V_1 \cap V_R\text{ and }b \in V_2 \cap V_B.$$  
(Indeed, if say $V_1$ only meets $V_R$ then $V_2$ necessarily meets  $V_B$, and the case that $V_1$ meets both sets is easy.)
We will define a red tree $T_1$ and a blue tree $T_2$ having $r$ and $b$ as their respective roots (and later, depending on the structure of the colouring, we might define a third tree $T_3$). To define $T_1$ and $T_2$, we will define a  function $\rho$ that assigns to  certain vertices $v$ a colour~$\rho(v)$. We only assign $v$ colour $\rho(v)$  if there are many monochromatic paths in colour $\rho(v)$ connecting $v$ with the root of the tree in colour $\rho(v)$. Later, $T_1$ will consist of all vertices $v$ with $\rho(v)=\red$ and $T_2$ will consist of all vertices $v$ with $\rho(v)=\blue$.

So let us define $\rho$. We set $\rho(r)= \red$ and $\rho(b)= \blue$, and   set
\[\rho(v) = \begin{cases} 
     \red & \text{if} \quad v \in N_R(r),\\
     \blue & \text{if} \quad v \in N_B(b). 
   \end{cases}
\]
Since $r\in V_R$ and $b\in V_B$, we have $|N_R(r)|,|N_B(b)|\ge {pn}/{100}$, which, together with~\ref{lem:ii}, implies  that $e(N_R(r), N_B(b)) \geq p|N_R(r)| |N_B(b)|/2$. By colour symmetry, we may assume without loss of generality that
 \begin{align}\label{eq:red:edges}
  e_R(N_R(r),N_B(b)) \geq \frac{p}{4} |N_R(r)| |N_B(b)|.
\end{align}
Let $$J_1=\{ v \in  N_B(b): |N_R(v) \cap N_R(r)| > {p^2n}/{25}\}.$$ The letter $J$ stands for the word `joker' as the vertices from $J_1$ are quite versatile: Note that each vertex from $J_1$ can be connected to $r$ by a red path of length two and to $b$ by a blue edge.  We claim that $J_1$ is large, namely,
\begin{equation}
    \text{\label{claim:Joker:large} $|J_1|\ge pn/100$.}
\end{equation}
To see~\eqref{claim:Joker:large}, we start by 
setting $L:=N_B(b)\setminus J_1$. Then
\[e_R(L,N_R(r))\leq |L|\cdot\frac{p^2n}{25}\le 
|N_B(b)|\cdot\frac{(1-\varepsilon)p^2n}{24}
\le \frac{p}{8} |N_R(r)| |N_B(b)|,\]
where we used~\ref{lem:i} and the fact that $r\in V_R$, which implies that $|N_R(r)| \geq (1-\eps)pn/3$.
Thus, by~\eqref{eq:red:edges}, we have 
\[e_R(J_1,N_R(r))\ge p|N_R(r)| |N_B(b)|/8.\] 
As $|N(v)\cap N(r) | \le (1+\eps)p^2n$ for every $v\in V$, we conclude that 
\[|J_1|\ge\frac{p|N_R(r)| |N_B(b)|}{(1+\eps)8p^2n}\ge \frac{pn}{100},\]
where in the last inequality we used that $d_R(r),d_B(b) \geq (1-\eps)pn/3$.  This proves~\eqref{claim:Joker:large}.

Next, set $$Z_2= \{ z \in V_2 \setminus (N_R(r) \cup \{b\} ) : |N(z,J_1)| \geq {p^2n}/{200}\}$$ and set  $K_2= V_2 \setminus (N_R(r) \cup Z_2 \cup \{b\})$. 
Because of~\ref{lem:iii} and~\eqref{claim:Joker:large}, we have 
 \begin{align}\label{K2}
  |K_2| & \leq {100}/{p}.
\end{align}
For  $z \in Z_2$, we set
\[ \rho(z) = \begin{cases} \red & \text{if} \quad | N_R(z,J_1)| \geq \dfrac{p^2n}{400},\\
\blue & \text{otherwise.}
   \end{cases}\]
   Note that each vertex in $Z_2$ has at least ${p^2n}/{400}$ neighbours in $J_1$ in at least one of the colours. So, if $\rho(z) =\blue$, then $| N_B(z,J_1)| \geq p^2n/400$. 

Up to this point, we have assigned a colour to every vertex in $V(G)$ except for those in $V_1 \setminus ( N_B(b) \cup \{ r\})$ and in $K_2$. Moreover, every vertex $v$ sends an edge of colour $\rho(v)$  to either $b$ or $r$ (according to $\rho(v)$), or many edges of the same colour to $J_1$. In order to be able to define $T_1$ and $T_2$ as disjoint trees, we will need to assign colours to the vertices of $J_1$ so that the vertices from $Z_2$ can all be connected with monochromatic paths to either $r$ or $b$.
For this reason, we will define $\rho (v)$ for the vertices of $J_1$ as follows:
\begin{align*}
&\text{For each $v\in J_1$, define $\rho(v)$ by choosing a colour from $\{\red,\blue\}$}\\ &\text{independently and uniformly at random.}
\end{align*}
Let $Z_2^r$ be the set of all vertices $v\in Z_2$ with $\rho(v)=\red$, and set $Z_2^r=Z_2\setminus Z_2^r$. Note that for each $v \in Z_2^r$, the probability that $\rho(w)=\blue$ for each vertex  $w\in N_R(v,J_1)$  is
\[2^{-|N_R(v,J_1)|} \leq 2^{-p^2n/400}=o(n^{-1}).\]
The same argument holds for  $v \in Z^b_2$, which proves  that 
\begin{equation}\label{Z2good}
     \text{w.h.p.~each $v \in Z_2$ has at least one neighbour of colour $\rho (v)$ in $J_1$.}
\end{equation}
By~\eqref{Z2good}, there is a red tree 
 $T_1$ with vertex set $\rho^{-1}(\red)$, such that the vertices of~$Z_2^r$ are leaves of $T_1$.  We define a blue tree  $T_2$ analogously, with $V(T_2) = \rho^{-1}(\blue)$. Note that by our choice of $\rho$, 
\begin{equation}
  \text{$T_1$ and $T_2$ are disjoint and
cover all of 
$( N_B(b) \cup \{ r\})\cup (V_2\setminus K_2)$. }
\end{equation}
We still need to cover most of the vertices in $V_1 \setminus(N_B(b) \cup \{r \} )$. For this, we will consider two cases.\smallskip

\noindent\textbf{Case 1.} There is a vertex $\tilde v \in V_1 \setminus (N_B(b) \cup \{r \})$ such that 
\[|d_B(\tilde v,Z_2^r)| \geq {pn}/{100}\quad\text{or}\quad|d_R(\tilde v,Z_2^b)| \geq {pn}/{100}.\]
We will assume that $|N_B(\tilde v, Z_2^r)| \geq {pn}/{100}$, as the other case is completely analogous (with all colours switched). So for $J_2= N_B(\tilde v,Z_2^r)$ we have $$|J_2| \geq {pn}/{100}.$$ We will use $\tilde{v}$ as the root of a third monochromatic tree. This tree will be blue, and in order to define it, we will use a function $\rho'$.

Let $Z_1= \{ x \in V_1 \setminus (N_B(b) \cup \{r, \tilde{v} \} ) : |N(x,J_2)| \geq {p^2n}/{200}\}.$
    We  define, for each $x\in Z_1$,
 \[\rho'(x) =  \begin{cases} 
      \red & \text{if }| N_R(x)\cap J_2| \ge \frac{p^2n}{400}, \\
      \blue & \text{otherwise.}
    \end{cases}\] 
    As every $x \in Z_1$ has at least $\frac{p^2n}{400}$ neighbours in $J_2$ having the same colour, vertices with blue assignment have many blue neighbours in $J_2$.

    Now, 
for each $v\in J_2$, we choose $\rho'(v)\in\{\red,\blue\}$ uniformly at random, making all choices independently from each other. Similarly as above, 
\begin{equation}\label{Z1good}
     \text{w.h.p.~each $x \in Z_1$ has at least one neighbour of colour $\rho' (x)$ in $J_2$.}
\end{equation}
Let $K_1= V_1 \setminus (N_B(b) \cup Z_1 \cup \{r, \tilde{v}\})$, and note that by~\ref{lem:iii} we have  
$|K_1|\le {100}/{p}.$
So, by~\eqref{K2}, we know that 
\[|K_1\cup K_2|\le \frac{200}{p}.\] 
We let $T_3$ be a blue tree with vertex set $\{\tilde v\}\cup \{v\in J_2\cup Z_1:\rho'(v)=\blue\}$. We let
$T'_1$ be the tree obtained from $T_1-\{v\in J_2:\rho'(v)=\blue\}$ by adding  as leaves all $x\in Z_1$ with $\rho' (x)=\red$. 
Then the three trees $T'_1$, $T_2$ and $T_3$ are vertex-disjoint and cover all of $V(G)\setminus (K_1\cup K_2)$, that is, they cover all but $200/p$ vertices from $G$, which is as desired.\\

\noindent\textbf{Case 2.} For every vertex $v \in V_1 \setminus (N_B(b) \cup \{r\})$, we have 
\[d_B(v,Z_2^r) \le\frac{pn}{100}\quad \text{and}\quad d_R(v,Z_2^b) \le \frac{pn}{100}.\]
Consider any vertex  $v \in  V_1 \setminus (N_b(b) \cup \{r \})$.
Note that by~\ref{lem:i}, \eqref{eq:pbound} and~\eqref{K2},  
   \begin{equation}\label{eq:degree:case2}d(v, Z_2)\ge d(v)-d(v,K_2)-d(v,N_R(r))-1\ge \frac{pn}{2},\end{equation} 
and therefore, $v$ either has at least $pn/10$ neighbours in $Z_2^r$ or at least $pn/10$ neighbours in $Z_2^b$. In the former case, we have 
$d_R(v,Z_2^r) \geq {pn}/{100}\ge 1,$
as $d_B(v,Z_2^r) \leq {pn}/{100}$.
In the latter case, we have $ d_B(v,Z_2^b) \geq pn/100\ge 1$. 

So, we can add each vertex  $v \in  V_1 \setminus (N_b(b) \cup \{r \})$ as a leaf to one of the trees $T_1$, $T_2$, and the obtained two trees 
cover all but at most $100/p$ vertices of $G$.\end{proof}

\section{Graphs of large minimum degree}
\label{sec:m}

The whole section is devoted to the 
proof of Theorem~\ref{theorem:bipartite}.

Let $n_0$ be sufficiently large and let $G$ be a bipartite graph with parts $V_1$ and $V_2$, each of size $n\ge n_0$, and with minimum degree $\delta(G)\ge ({13}/{16}+\delta)n$. Let $V_R$ and $V_B$ be the set of vertices with at least $({9}/{16}+{3\delta}/4)n$ neighbours in red and blue, respectively. If one of these sets is empty, say $V_B$, then the red graph has minimum degree at least 
\[\left(\frac{13}{16}+\delta\right)n-\left(\frac{9}{16}+\frac{3\delta}{4}\right)n>n/4,\]
and thus $G$ can be covered using at most $3$ red components, and we are done. So we can assume both sets $V_R$, $V_B$ are non-empty.
This implies that there are vertices $r\in V_R$ and $b\in V_B$ from different partition classes of $G$. 

Suppose without loss of generality that $r\in V_1$ and $b\in V_2$, and   choose sets $X\subset N_R(r)\setminus\{b\}$ and $Y\subset N_B(b)\setminus\{r\}$ such that $|X|=|Y|=({9}/{16}+{\delta}/{2})n$. Note that because of our condition on $\delta(G)$, for every vertex $v\in Y$ we have
\[ d(v,X)\ge \left(\frac{9}{16}+\frac{\delta}{2}\right)n-\left(\frac{3}{16}-\delta\right)n\ge \left(\frac 38+\delta\right)n,\]
and thus we have $e(X,Y)\ge(3/8+\delta)n|Y|.$
Therefore, there are at least $({3}/{16}+{\delta}/{2})n|Y|$ edges of the same colour between $X$ and $Y$. Say this colour is red, and let $J_Y\subset Y$ be the set of all vertices $v\in Y$ satisfying $d_R(y,X)\ge \delta n/100$. (If the most popular colour between $X$ and $Y$ was blue, then we take $J_X\subset X$ instead, and accordingly change the rest of the proof.) 

We claim that $|J_Y|\ge {3n}/{16}$. Indeed, otherwise, we would have that 
\[\left(\frac{3}{16}+\frac{\delta}{2}\right)n|Y|\le e_R(X,Y)\le \frac{3n}{16}\cdot |X|+|Y|\cdot \frac{\delta n}{100}, \]
which is a contradiction as $|X|=|Y|$.  Therefore, and because of our condition on $\delta(G)$, every vertex $v\in V_2$ satisfies 
\begin{equation}\label{equation:degree:jk}d(v,J_Y)\ge \frac{3n}{16}-\left(\frac{3n}{16}-\delta\right)n=\delta n.\end{equation}
Before starting assigning colours to vertices as in the proof of Theorem~\ref{thm:almost-cover}, we need to shrink the size of $X$ for reasons that will become clear later in the proof. Let $p\in  (0,1)$ be a small enough constant and take a subset $X'\subset X$ so that each vertex $x\in X$ is included in $X'$ with probability $p$ and all choices are made independently. Note that $\mathbb E|X'|=p|X|=
({9}/{16}+\delta/2)pn$
 and, by definition of $J_Y$, each vertex $y\in J_Y$ satisfies $\mathbb E [d_R(y,X')]\ge {\delta pn}/{100}.$
So, by Lemma~\ref{lemma:chernoff}, with probability $1-o(1)$ we have that
\begin{enumerate}[label=$\bullet$]
\item $\frac{pn}{2}\le |X'|\le pn$, and
    \item for each vertex $y\in J_Y$, $d_R(y,X')\ge \delta pn/200$.
\end{enumerate} 
Let $W=V_2\setminus (\{b\}\cup X')$. Now, as in the proof of Theorem~\ref{thm:almost-cover}, we will choose a preferred colour $\rho(v)$ for each vertex~$v$.  We set 
\[\rho(v)=\begin{cases}\red &\text{if }v\in X'\cup \{r\}, \text{ and}\\
\blue &\text{if }v\in (Y\cup \{b\})\setminus J_Y,
\end{cases}\]
and  for each $v\in J_Y$, we choose $\rho(v)\in \{\red,\blue\}$ uniformly at random. 
Let us next define $\rho(v)$ for each $v\in W$. Because of~\eqref{equation:degree:jk}, for each  $v\in W$ there is a colour $c_v\in\{\red,\blue\}$ such that $v$ has at least $\delta n/2$ neighbours in $J_Y$ in colour $c_v$, in which case we let $\rho(v)=c_v$. 

Use Lemma~\ref{lemma:chernoff} and a union bound to see that the following property holds with high probability.
\begin{enumerate}[label=]
    \item For each vertex $w\in W$ there are at least $\delta n/8$ vertices $v\in N(w,J_Y)$ in colour $\rho(w)$ such that $\rho (v)=\rho(w)$.
\end{enumerate} 
Set $W_B=\rho^{-1}(\blue)\cap W$ and $W_R=\rho^{-1}(\red)\cap W$, and observe that every vertex in $W_B$ (resp. $W_R)$ can be connected to $b$ (resp. to $r)$ by a path in colour blue (resp. red). As $|X'|\le pn$ and $p$ is sufficiently small, we have $|W|=n-|X'|-1\ge n-2pn\ge 0.9n$, and thus, at least one of $W_R$, $W_B$ has at least $0.4n$ vertices. We assume $|W_B|\ge 0.4n$ (the other case is analogous).

Let $U=V_1\setminus (\{r\}\cup Y)$, and observe that every vertex $v\in U$ satisfies 
\begin{equation}\label{dvWb}
    d(v,W_B)\ge 0.4n-\left(\frac{3}{16}-\delta\right)n\ge \left(\frac{3}{16}+\delta\right)n.
\end{equation}
If each vertex in $U$ has at least one blue neighbour in $W_B$, we can assign $\rho(u)=\blue$ for each $u\in U$. Then both $\rho^{-1}(\red)$,  $\rho^{-1}(\blue)$ span vertex-disjoint monochromatic connected components, and we are done. So, we can assume that there is a vertex $u_0\in U$ with no blue neighbour in $W_B$ and thus $d_R(u_0,W_B)\ge (3/16+\delta)n$ by~\eqref{dvWb}.

Pick a set $J_X\subset N_R(u_0,W_B)$ of size $(3/16+\delta)n$, and note that each vertex $u\in U$ satisfies
\[d(u,J_X)\ge (3/16+\delta)n-(3/16-\delta)n\ge 2\delta n.\]
In particular, for each vertex $u\in U$ there is a colour $c_u$ such that $u$ has at least $\delta n$ neighbours in colour $c_u$ in $J_X$. We set $\rho(u)=c_u$. Now, we randomly re-colour each vertex in $J_X$ by choosing $\rho(v)\in\{\red,\blue\}$ uniformly at random for each $v\in J_X$. Since for each $u\in U$, the expected number of vertices in $J_X$ with colour $\rho(u)$ is at least $\delta n$,  by Lemma~\ref{lemma:chernoff} we deduce that w.h.p.~each vertex $u\in U$ has at least $\delta n/2$ neighbours $v\in J_X$ such that $\rho(v)=\rho(u)$. 
So, the vertices with preference $\rho(v)=$ red span a red subgraph having at most two connected components, and the remaining vertices (those with with preference $\rho(v)=$ blue) span a blue connected subgraph, and we are done.

\section{Concluding remarks}

\subsection{Variants of the Gy\'arf\'as-Lehel conjecture}\label{sec:gyalehvar}

Recall that in Conjecture~\ref{conjecture:GL}, 
Gy\'arf\'as and Lehel 
conjectured that
$\tc_r(K_{n,m})\le 2r-2$ for $n,m\ge 1$ and $r\ge 2$. In~\cite{chen2012around} and~\cite{gya77} one can find examples of $r$-colourings of $K_{n,m}$ with $n=r-1$ and $m=r!$ which show that the bound of $2r-2$ in Conjecture~\ref{conjecture:GL} is best possible.

We do not know of similar examples for balanced complete bipartite graphs. That is, we do not know whether there is,  for some $n$, an $r$-colouring of $K_{n,n}$ which does not allow for a  covering with less than $2r-2$ trees. So, although we are inclined to believe the opposite, it is possible that $\tc_r(K_{n,n})$ is lower than $2r-2$. Thus motivated, we propose the following problem.
\begin{problem}\label{prob:unbal}
    Determine $\tc_r(K_{n,n})$ for all $n,r$. 
\end{problem}
Also, it seems that nothing is known about tree partitioning in complete bipartite graphs, which is why we suggest the following problem.
\begin{problem}
    Determine $\tp_r(K_{n,m})$ for all $n,m,r$. 
\end{problem}
We believe that in analogy to the complete graph case, it could be true that $\tp_r(K_{n,m})=\tc_r(K_{n,m})$.
\subsection{Unbalanced random bipartite host}\label{sec:unbalRandom}

While we believe the bipartite tree covering number $\tc_r(K_{n,m})$ should be the same if restricted to the instances where $n=m$, this might not be the case. Then,
our main result, Theorem~\ref{thm:almost-cover} corresponds to Problem~\ref{prob:unbal}, and the random version of Conjecture~\ref{conjecture:GL} should be more general. Nevertheless, our proof, up to minor modifications, yields the same conclusions for the host $G(n,m,p)$ when $m=\Theta(n)$.
 \subsection{More colours in the random setting}\label{sec:moreColours}
 
Given Theorem~\ref{thm:almost-cover}, a natural next step would be to find a random analogue of the Gy\'arfas--Lehel conjecture for more colours. In this direction, we propose the following question.

\begin{question}\label{question:covering:r}For $r\ge 2$, determine the threshold for $\tc_r(G(n,n,p))\le 2r-1$.
\end{question}

As in the two colour case, we believe that the answer to Question~\ref{question:covering:r} should be the same as in the non-bipartite setting. In particular, given the result of Bradra\v{c} and Buci\'c~\cite{bradavc2023covering}, we believe that the threshold for 3 colours should be $p=(\log n/n)^{1/4}$. 
\begin{conjecture}There exists a constant $C$ such that if $p\ge C(\log n/n)^{1/4}$ and $G\sim G(n,n,p)$, then w.h.p.~$\tc_3(G)\le 5$.
\end{conjecture}


\bibliographystyle{abbrv}
\bibliography{ref}

\appendix
\section{Proof of Lemma~\ref{lemma:prelim1}}
We let  $C\gg \varepsilon^{-2}$.
\begin{enumerate}[label=(\roman{enumi}):]
    \item   As $d(u)\sim \bin(n,p)$ for every  $u\in V_1$, we can use Chernoff's bound (Lemma~\ref{lemma:chernoff}), to see that the probability that there is a vertex $u$ with $|d(u)-pn|\ge\varepsilon pn$ is bounded by 
    $2ne^{-\varepsilon^2pn/3}=o(1).$
    Moreover, for  distinct  $u,v\in V_1$, we have that $|N(u)\cap N(v)|\sim \bin(n,p^2)$ and thus, by Chernoff's bound, the probability that are such $u,v$ with  $\left||N(u)\cap N(v)|-p^2n\right|\ge \varepsilon pn^2$ is bounded by
    $2n^2e^{-\varepsilon^2 p^2n/3}=o(1).$

    \item Fix $U\subset V_1$ and $W\subset V_2$, with $|U|\ge 100/p$ and $|W|\ge pn/100$. Note that $e(U,W)\sim \bin(|U||W|,p)$ and so, by Chernoff's bound (Lemma~\ref{lemma:chernoff}), we have 
    \[\mathbb P(e(U,W)<p|U||W|/2)\le e^{-p|U||W|/8}.\]
    Therefore, the probability that there exist a vertex $v\in V_2$, and subsets $U\subset N(v)$ and $W\subset V_2$, with $|U|\ge pn/100$ and $|W|\ge 100/p$, such that $e(U,W)<p|U||W|/2$, is bounded by  
    \begin{eqnarray*}\sum_{u\ge \frac{pn}{100} , w\ge\frac{100}p}n\binom{n}{u}\binom{n}{w}p^ue^{-puw/8}&\le& \sum_{{u\ge \frac{pn}{100} , w\ge\frac{100}p}}n\left(\frac{enp}{u}\right)^{u}\left(\frac{en}{w}\right)^{w}\\
    &\le&\sum_{{u\ge \frac{pn}{100} , w\ge\frac{100}p}}ne^{5u+w\log n-puw/8}\\
    &\le&\sum_{{u\ge \frac{pn}{100} , w\ge\frac{100}p}}ne^{-uw/100}\\&\le& n^3e^{-n/100}=o(1),\end{eqnarray*}
    where we used that $puw/20\ge 5u$ and $puw/20\ge pn^2w/2000\ge  w\log n$.
    
    \item Let $W$ be the set of those vertices $w\in V_2$ with $d(w,U)<p^2n/200$. Since   by (ii),
    \[e(U,W)=\sum_{w\in U}d(w,U)<pn|W|/200\le p|U||W|/2,\]
 we deduce that $|W|<100/p$.
 
  \item Fix $0<\delta\ll \varepsilon$. We first note that $e(A,B)\le p|A||B|+\delta pn|A|$ for all sets $A$ and $B$ from opposite bipartition classes, both of size at least $pn/2$. Indeed, given such sets $A$ and $B$, by Lemma~\ref{lemma:chernoff} and since $pn\gg 1/p$, we have that  
  \[\mathbb P(e(A,B)\ge p|A||B|+\delta pn|A|)\le e^{-(\delta n/|B|)^2p|A||B|/2}\le e^{-\delta^2pn|Y|/2}=o(2^{2n}).\]

  Let $U$ be a component of $H$. Set $X=V_1\cap U$ and $Y=V_2\cap U$. By (i) we have $|X|,|Y|\ge \delta (H)\ge pn/2$, and therefore, 
  \[p|X||Y|+\delta pn|X|\ge e_H(X,Y)=\sum_{x\in X}d_H(x)\ge (1/2+\varepsilon)pn|X|.\]
  Hence $|Y|\ge (1/2+\varepsilon/2)n$, and similarly, we see that $|X|\ge (1/2+\varepsilon/2)n$. As this is true for any component $U$, $H$ is connected. 
\end{enumerate}
\end{document}